\documentclass[12pt]{amsart}
\usepackage{amsmath,amsfonts,euscript,dsfont,amscd,amsthm,amssymb,upref,graphics}
\usepackage[all]{xy}
\usepackage{color}

\theoremstyle{definition}

\swapnumbers

\theoremstyle{plain}

\newtheorem{theorem}{Theorem}[section]

\newtheorem{proposition}[theorem]{Proposition}
\newtheorem{lemma}[theorem]{Lemma}
\newtheorem{corollary}[theorem]{Corollary}

\theoremstyle{definition}

\newtheorem{definition}[theorem]{Definition}

\newtheorem{parag}[theorem]{}
\newtheorem{example}[theorem]{Example}

\newtheorem{notations}[theorem]{Notations}

\newtheorem{remark}[theorem]{Remark}

\theoremstyle{remark}

\newtheorem*{smallremark}{Remark}

{\begin{enumerate}\setlength{\itemsep}{#1}}{\end{enumerate}}
\newenvironment{enumerata}%
{\begin{enumerate}

}{\end{enumerate}}
{\begin{enumerata}\setlength{\itemsep}{#1}}{\end{enumerata}}

\newcommand{\setspec}[2]{\big\{\,#1\, \mid \,#2\, \big\}}

\newcommand{\Integ}{\ensuremath{\mathbb{Z}}}
\newcommand{\Nat}{\ensuremath{\mathbb{N}}}

\newcommand{\Comp}{\ensuremath{\mathbb{C}}}

\newcommand{\aff}{\ensuremath{\mathbb{A}}}
\newcommand{\proj}{\ensuremath{\mathbb{P}}}

\newcommand{\Aeul}{\EuScript{A}}

\newcommand{\Ceul}{\EuScript{C}}
\newcommand{\Deul}{\EuScript{D}}

\newcommand{\Neul}{\EuScript{N}}

\newcommand{\Teul}{\EuScript{T}}

\newcommand{\Veul}{\EuScript{V}}

\renewcommand{\epsilon}{\varepsilon}
\renewcommand{\phi}{\varphi}
\renewcommand{\emptyset}{\varnothing}

\newlength{\mylength}
\newcommand{\rien}[1]{}

\CompileMatrices

\begin{document}
\renewcommand{\baselinestretch}{1.07}


\title[Rational polynomials of simple type: a combinatorial proof]%
{Rational polynomials of simple type:\\ a combinatorial proof}

\author{Pierrette Cassou-Nogu\`es}
\author{Daniel Daigle}

\address{IMB, Universit\'e de Bordeaux \\
351 Cours de la lib\'eration, 33405, Talence Cedex, France}
\email{Pierrette.Cassou-nogues@math.u-bordeaux1.fr}

\address{Department of Mathematics and Statistics\\
University of Ottawa\\
Ottawa, Canada\ \ K1N 6N5}
\email{ddaigle@uottawa.ca}

\thanks{Research of the first author partially supported by Spanish grants
MTM2013-45710-C02-01-P and  MTM2013-45710-C02-02-P}
\thanks{Research of the second author supported by grant RGPIN/104976-2010 from NSERC Canada.}

{\renewcommand{\thefootnote}{}
}

{\renewcommand{\thefootnote}{}
}

\begin{abstract} 
We determine the Newton trees of the rational polynomials of simple type, thus filling a gap in 
the proof of the classification of these polynomials given by Neumann and Norbury.
\end{abstract}

\maketitle
  
\vfuzz=2pt

\section{Introduction}

A polynomial map $f:\Comp^2 \to \Comp$ is {\it rational\/} if its generic fiber is of genus zero.
It is {\it of simple type\/} if, when extended to a morphism $\overline{f}:X \to \proj^1$ of a compactification $X$ of $\Comp ^2$,
the restriction of $\overline{f}$ to each curve of the divisor $D=X\setminus \Comp^2$ is either of degree $0$ or $1$.
The curves  on which $\overline{f}$ is non constant are called {\it dicriticals}.
The  {\it degree}  of a dicritical $C$ is the degree of the restriction of  $\overline{f}$ on $C$.
A {\it simple} rational polynomial is a rational polynomial all of whose dicriticals have  degree $1$.
We say that a polynomial is {\it ample} if it has at least $3$ dicriticals of degree $1$.

A classification of rational polynomials of simple type appears in \cite{MiySugie:GenRatPolys}, but is incomplete.
This has been noticed by  Neumann and Norbury, who have presented a new classification in \cite{NeumannNorbury:simple}.
Actually, the larger part of  \cite{NeumannNorbury:simple} is devoted to the
determination of the possible splice diagrams of rational polynomials that are ample and of simple type;
at the end, they give the equations of the polynomials having these splice diagrams.

However, the proof given in \cite{NeumannNorbury:simple} appears to be incomplete.
In the first paragraph after the proof of Lemma 3.1,
the authors construct a rational map $\pi: X \to \proj^1$, where $X$ is a compactification of $\Comp^2$.
Note that the rational map $\pi$ that they consider is not defined by extending a morphism $\Comp^2 \to \proj^1$,
and that there are no obvious reasons to think that $\pi$ is defined at all points of $\Comp^2$.
Then the authors write:
``If $\pi$ is not a morphism then we blow up $X$ to get a morphism.
Rather than introducing further notation for this blow-up we will assume we began with this blow-up and call it $X$.''
The last sentence contains the hidden assumption that the new $X$ is still a compactification of $\Comp^2$, as the old $X$ was.
In other words, it is implicitly claimed that all points of indeterminacy of the original rational map $\pi$
belong to  $X \setminus \Comp^2$.
That implicit claim is not proved in \cite{NeumannNorbury:simple}, and in fact we don't know if it is true.
Since the determination of the splice diagrams crucially depends on that implicit claim,
there is a gap in the proof of the main result of \cite{NeumannNorbury:simple}.

The present article studies a type of combinatorial object that we call ``abstract Newton tree at infinity''.
These are abstract decorated trees and are closely related to splice diagrams.

Section 2 develops a general theory of abstract Newton trees at infinity; the main result of that section is Theorem~\ref{918235071yrsj2dhry}.
Section 3 is devoted to describing all Newton trees whose multiplicity is maximal;
these trees are described by Theorem~\ref{8c3r8273d287ed2uq98}, which is the main result of the article.
Note that Sections 2 and 3 are purely combinatorial and constitute the core of the paper.

Section 4 applies Theorem~\ref{8c3r8273d287ed2uq98} to the problem of classifying rational polynomials of simple type.
The first part of that section 
describes a process that associates an abstract Newton tree at infinity $\Teul(F)$ to each primitive polynomial $F \in \Comp[x,y]$ with at least two points at infinity.
The process is in two steps: $F$ determines a family $\big( \Omega_\lambda \big)_{\lambda \in \Comp}$ of splice diagrams, and this family determines
the Newton tree $\Teul(F)$.

Prop.\ \ref{jkcnvo2ws9cj} states that if $F$ is a rational polynomial of simple type that is not a variable then
$\Teul(F)$ satisfies the hypothesis of Theorem~\ref{8c3r8273d287ed2uq98}
(so $\Teul(F)$ is among the trees given by that result).

At the end of Section 4 we observe that the Newton trees given by Theorem~\ref{8c3r8273d287ed2uq98} correspond exactly to the splice diagrams
given by Neumann and Norbury, except for the fact that we don't restrict ourselves to the case of ample polynomials, as Neumann and Norbury did.
We may therefore confirm that the splice diagrams given in \cite{NeumannNorbury:simple} are correct---this needed
to be ascertained, because of the gap in the proof that we pointed out in the above discussion.

More generally, we are interested in the problem of classifying the Newton trees $\Teul(F)$ of rational polynomials $F$ that are not variables.
The present paper settles the special case where $F$ is of simple type.
In a subsequent article we use the same approach and classify the trees of rational polynomials of quasi-simple type
(note that Sasao classifies a subclass of the rational polynomials of quasi-simple type in \cite{Sasao_QuasiSimple2006}).
Our method is very elementary and we hope to use it to solve more cases of the classification problem.

\section{Abstract Newton trees at infinity}

This section and the next develop the theory of abstract Newton trees at infinity in a purely combinatorial manner.
Readers who want to see how these notions are related to splice diagrams (and hence to geometry) are invited to
consult Section~4.

\begin{definition} \label {p9823p98p2d}
Consider a rooted tree $\Teul$
with two kinds of 0-dimensional cells called vertices and arrows,
and 1-dimensional cells called edges, where each edge links two distinct 0-dimensional cells.
We say that $\Teul$ is an
{\it abstract Newton tree at infinity\/} if all of the following conditions are satisfied,
including conditions (1) and (2) below.

There are finitely many vertices, arrows, and edges.
We denote by $\Veul$ the set of vertices and by $\Aeul$ the set of arrows,
and we assume that $\Veul \neq \emptyset$, $\Aeul \neq \emptyset$ and $\Veul \cap \Aeul = \emptyset$.
The root, denoted by $v_0$, belongs to $\Veul$. 
Given $v\in \Veul\cup \Aeul$, the number of edges incident to $v$ is called its {\it valency\/}
and is denoted by $\delta_v$.
The valency of each $v\in \Veul \setminus \{v_0\}$ is at least $2$, and the valency of each $v\in \Aeul$ is~$1$.

By a {\it path}, we always mean a simple path, i.e., a path that does not traverse the same edge more than once.

There is a partial order on $\Veul\cup \Aeul$:
given distinct elements $v,v'$ of $(\Veul\cup \Aeul) \setminus \{v_0\}$,
we say that $v'>v$ if $v$ is on the path between $v_0$ and $v'$.
Moreover, if $v \in (\Veul\cup \Aeul)\setminus \{v_0\}$ then $v>v_0$.

The edges and the arrows are decorated.
\begin{enumerate}

\item The arrows bear decorations $(0)$ or $(1)$; when the decoration of an arrow does not appear in a picture,
that decoration is assumed to be $(1)$. We denote by $\Aeul_0$ the set of arrows decorated with $(0)$. An edge  linking a vertex and  an arrow decorated with $(0)$ is called a {\it dead end}. 
Moreover, the following are required to hold.
\begin{enumerate}

\item For each vertex $v \in \Veul$, there exists an arrow $\alpha$ decorated with $(1)$ such that $\alpha>v$.
\item There is at most one dead end incident to a given vertex.
\end{enumerate}

\item The edges bear at each extremity a decoration which is an element of $\mathbb{Z}$;
when a decoration does not appear in a picture, it is assumed to be $1$. 
If $e$ is an edge incident to a vertex or arrow $v$,
the decoration of $e$ near $v$ is said to be a ``decoration near $v$''.
We denote this number by $q(e,v)$.
The following are required to hold.

\begin{enumerate}

\item All decorations near the root are $1$.

\item The decoration near any arrow is $1$.

\item \label {1823urcpiwd} 
Let $v\in \Veul$. If $e,e'$ are distinct edges incident to $v$ then 
$q(e,v)$ and $q(e',v)$ are relatively prime.
Let $E$ be the set of edges 
of the form $[v,v']$ where $v' \in \Veul \cup \Aeul$ and $v<v'$.
Then $E \neq \emptyset$, $q(e,v) \ge1$ for all $e \in E$ and at most one element $e$ of $E$ satisfies $q(e,v)>1$.
If $\epsilon \in E$ is such that $q(\epsilon,v) = \max_{e \in E} q(e,v)$, we call $\epsilon$ a {\it leading edge\/} at $v$
(it may be that every element of $E$ is leading).
If there is a dead end $e$ incident to $v$ then $e$ is a leading edge at $v$.

\item Let $e$ be an edge between two vertices $u$ and $v$.
The decorations near $u$ or $v$ but not on $e$ are said to be adjacent to $e$.
The {\it edge determinant\/} of $e$ is the product of the decorations on $e$,
minus the product of the decorations adjacent to $e$. All edge determinants are required to be negative.

\end{enumerate}

\end{enumerate}
\end{definition}

\begin{example}
Here is an abstract Newton tree at infinity with $4$ vertices, $5$ arrows, and $8$ edges:
$$
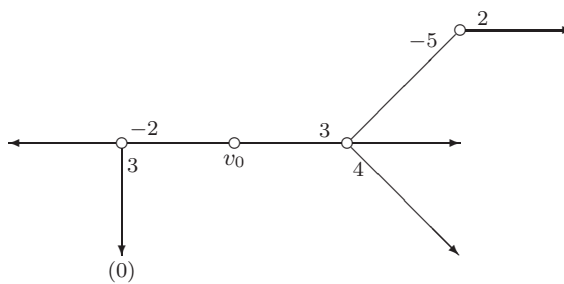
\setlength{\unitlength}{1.5mm}
\begin{picture}(60,24)(-25,-13)
\put(-10,0){\circle{1}}
\put(0,0){\circle{1}}
\put(10,0){\circle{1}}
\put(20,10){\circle{1}}
\put(-10.5,0){\vector(-1,0){9.5}}
\put(-10,-.5){\vector(0,-1){9.5}}
\put(10.5,0){\vector(1,0){9.5}}
\put(10.4,-.4){\vector(1,-1){9.5}}
\put(20.5,10){\vector(1,0){9.5}}
\put(10.4,0.4){\line(1,1){9.2}}
\put(-9.5,0){\line(1,0){9}}
\put(.5,0){\line(1,0){9}}
\put(0,-1){\makebox(0,0)[t]{\scriptsize $v_0$}}
\put(-10,-10.5){\makebox(0,0)[t]{\tiny $(0)$}}
\put(-8,.5){\makebox(0,0)[b]{\tiny $-2$}}
\put(-9.5,-2){\makebox(0,0)[l]{\tiny $3$}}
\put(8,.5){\makebox(0,0)[b]{\tiny $3$}}
\put(11.5,-1.7){\makebox(0,0)[rt]{\tiny $4$}}
\put(18,9){\makebox(0,0)[r]{\tiny $-5$}}
\put(22,10.5){\makebox(0,0)[b]{\tiny $2$}}
\end{picture}
$$
{\setlength{\unitlength}{1.5mm}
Note that each arrow is represented by an arrowhead
``\begin{picture}(2,1)(-1,-.5) \put(0.5,0){\vector(1,0){0}} \end{picture}'', not by an arrow
``\begin{picture}(5,1)(0,-.5) \put(0,0){\vector(1,0){5}} \end{picture}'',
so that ``\begin{picture}(5.5,1)(-.5,-.5) \put(0,0){\circle{1}} \put(0.5,0){\vector(1,0){4.5}} \end{picture}''
represents an edge joining a vertex 
``\begin{picture}(2,1)(-1,-.5) \put(0,0){\circle{1}} \end{picture}''
to an arrow
``\begin{picture}(2,1)(-1,-.5) \put(0.5,0){\vector(1,0){0}} \end{picture}''.}
We stress that the above picture completely defines an abstract Newton tree at infinity,
because all missing decorations of edges (resp.\ of arrows) are assumed to be $1$ (resp.\ $(1)$).
\end{example}

Consider an abstract Newton tree at infinity.

\begin{definition}
Let us agree that empty products of numbers are equal to $1$.
\mbox{\ }
\begin{enumerate}

\item[(i)] Given $v \in \Veul \cup \Aeul_0$ and an edge $e$ incident to $v$,
recall that $q(e,v)$ denotes the decoration of $e$ near $v$ and let
$Q(e,v) = \prod_{\epsilon \in E\setminus\{e\}} q(\epsilon, v)$, where $E$ is the set of edges incident to $v$.

\item[(ii)] Let $\gamma$ be a path. We say that an edge $\epsilon$ is {\it incident to $\gamma$} if $\epsilon$ is not in $\gamma$
and $\epsilon$ is incident to some vertex $u$ of $\gamma$.
In this case, we define $q(\epsilon,\gamma) = q(\epsilon,u)$, where $u$ is the unique vertex of $\gamma$ to which
$\epsilon$ is incident.

\item[(iii)] Let $v \neq v' \in \Veul$ and consider the path $\gamma$ from $v$ to $v'$.
We say that $\gamma$ is a {\it linear path\/} if each vertex $u$ in $\gamma$ but different from $v,v'$ has valency $2$.
If $\gamma$ is a linear path from $v$ to $v'$ then we define $\det(\gamma) = q(e,v) q(e',v') - Q(e,v)Q(e',v')$,
where $e$ (resp.\ $e'$) is the unique edge in $\gamma$ which is incident to $v$ (resp.\ $v'$).

\item[(iv)] Given $v \in \Veul \cup \Aeul_0$ and $\alpha \in \Aeul\setminus\Aeul_0$,
we set
$$
\textstyle
x_{v,\alpha} = \prod_{\epsilon \in E} q(\epsilon,\gamma)
\quad \text{and} \quad
\hat x_{v,\alpha} = \prod_{\epsilon \in \hat E} q(\epsilon,\gamma)
$$
where $\gamma$ is the path from $v$ to $\alpha$,
$E$ is the set of edges incident to $\gamma$ and
$$
\hat E =  \text{set of edges incident to $\gamma$ but not incident to $v$}.
$$
Observe that 
$x_{v,\alpha} = Q(e,v) \hat x_{v,\alpha}$,
where $e$ is the unique edge incident to $v$ which is in $\gamma$.

\item[(v)] Given $v\in \Veul\cup \Aeul_0$, we define the {\it multiplicity\/} $N_v$ of $v$ by
$ N_v = \sum_{ \alpha \in \Aeul \setminus \Aeul_0} x_{v,\alpha}$.

\item[(vi)] We define the {\it degree\/} of the tree to be $N_{v_0}$ where $v_0$ is the root.
Note that $N_{v_0} \ge1$ 
(proof: by \ref{p9823p98p2d}\eqref{1823urcpiwd} we have $x_{v_0,\alpha} \ge1$ for each $\alpha \in \Aeul\setminus\Aeul_0$,
so $ N_{v_0} = \sum_{ \alpha \in \Aeul \setminus \Aeul_0} x_{v_0,\alpha} \ge |\Aeul \setminus \Aeul_0|\ge1$).

\end{enumerate}
\end{definition}


\begin{proposition}  \label {kuwdhr12778}
Let $v \neq v' \in \Veul$, let $\gamma$ be a linear path from $v$ to $v'$,
let $q=q(e,v)$, $q'=q(e',v')$, $Q=Q(e,v)$ and $Q'=Q(e',v')$ 
where $e$ (resp.\ $e'$) is the unique edge in $\gamma$ which is incident to $v$ (resp.\ $v'$),
and let
\begin{align*}
A &= \setspec{ \alpha \in \Aeul\setminus\Aeul_0 }{ \text{the path from $v$ to $\alpha$ does not contain $v'$} }, \\
A' &= \setspec{ \alpha \in \Aeul\setminus\Aeul_0 }{ \text{the path from $v$ to $\alpha$ contains $v'$} } 
\end{align*}
(see Figure~\ref{83473yruer938d}).  Then the following hold.
\begin{enumerata}

\item \label {8723d4917t3rh} 
 For each $\alpha \in A$,
$x_{v,\alpha} = q \hat x_{v',\alpha}$
and $x_{v',\alpha} = Q' \hat x_{v',\alpha}$.

\item \label {u66d3swxsd}
For each $\alpha \in A'$, $x_{v',\alpha} = q' \hat x_{v,\alpha}$
and $x_{v,\alpha} = Q \hat x_{v,\alpha}$.

\item \label {238dhwi912F}
$\left| \begin{matrix} q & Q' \\ N_v & N_{v'} \end{matrix} \right| = 
\det(\gamma) \sum_{\alpha \in A'} \hat x_{v,\alpha}$\ \ and\ \ 
$\left| \begin{matrix} q' & Q \\ N_{v'} & N_{v} \end{matrix} \right| = 
\det(\gamma) \sum_{\alpha \in A} \hat x_{v',\alpha}$.

\item \label {61324d1g3r}
 If $v<v'$ then $q>0$, $Q'>0$, $\det\gamma<0$ and
$\left| \begin{matrix} q & Q' \\ N_v & N_{v'} \end{matrix} \right| < 0$.

\end{enumerata}
\end{proposition}

\begin{figure}[h]
\setlength{\unitlength}{1mm}
\raisebox{0\unitlength}{\begin{picture}(80,30)(-40,-15)
\put(-10,0){\circle{1}}
\put(10,0){\circle{1}}
\put(-9.5,0){\line(1,0){19}}
\put(-10.632,0.316){\line(-2,1){14.367}}
\put(-10.632,-.316){\line(-2,-1){14.367}}
\multiput(-11,-.333333)(-1,-.3333333){9}{\makebox(0,0){\footnotesize .}} 
\put(-20,-3.33333333){\vector(-3,-1){0}}
\put(-22,-4){\makebox(0,0){\tiny $(0)$}}
\multiput(11,-.33333333)(1,-.3333333){9}{\makebox(0,0){\footnotesize .}} 
\put(20,-3.33333333){\vector(3,-1){0}}
\put(22,-4){\makebox(0,0){\tiny $(0)$}}
\put(10.632, 0.316){\line(2,1){14.367}}
\put(10.632,-0.316){\line(2,-1){14.367}}
\put(-30,0){\oval(10,30)}
\put(30,0){\oval(10,30)}
\put(-30,0){\makebox(0,0){\footnotesize $A$}}
\put(30,0){\makebox(0,0){\footnotesize $A'$}}
\put(-15,0){\makebox(0,0)[r]{\footnotesize $Q$}}
\put(-8,-1){\makebox(0,0)[t]{\footnotesize $q$}}
\put(-10,1){\makebox(0,0)[b]{\footnotesize $v$}}
\put(0,1){\makebox(0,0)[b]{\footnotesize $\gamma$}}
\put(10,1){\makebox(0,0)[b]{\footnotesize $v'$}}
\put(8,-1){\makebox(0,0)[t]{\footnotesize $q'$}}
\put(15,0){\makebox(0,0)[l]{\footnotesize $Q'$}}
\end{picture}}
\caption{Schematic representation of the situation of Prop.~\ref{kuwdhr12778}.
The dotted arrows indicate that there may or may not be a dead end attached to $v$ (resp.\ to $v'$).}
\label {83473yruer938d}
\end{figure}
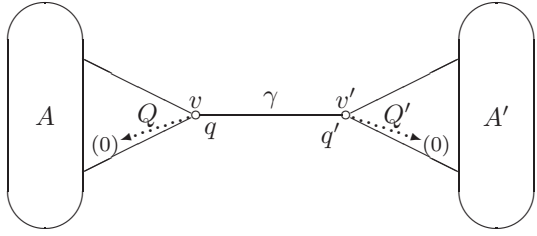

\begin{proof}
Verification of \eqref{8723d4917t3rh} is left to the reader, 
 and \eqref{u66d3swxsd} follows from \eqref{8723d4917t3rh} by symmetry.
Using \eqref{8723d4917t3rh}  and \eqref{u66d3swxsd}, we find
\begin{align*}
q N_{v'}
&= q \sum_{\alpha \in A} x_{v',\alpha} + q \sum_{\alpha \in A'} x_{v',\alpha}
= q Q' \sum_{\alpha \in A} \hat x_{v',\alpha} + q q' \sum_{\alpha \in A'} \hat x_{v,\alpha}
= Q' \sum_{\alpha \in A} x_{v,\alpha} + q q' \sum_{\alpha \in A'} \hat x_{v,\alpha},\\
Q' N_{v}
&= Q' \sum_{\alpha \in A} x_{v,\alpha} + Q' \sum_{\alpha \in A'} x_{v,\alpha}
= Q' \sum_{\alpha \in A} x_{v,\alpha} + Q Q' \sum_{\alpha \in A'} \hat x_{v,\alpha},
\end{align*}
so $q N_{v'} - Q' N_{v} = ( q q' - QQ') \sum_{\alpha \in A'} \hat x_{v,\alpha}$,
which proves the first part of assertion~\eqref{238dhwi912F}. The second part of assertion~\eqref{238dhwi912F}
then follows by symmetry.

\noindent\eqref{61324d1g3r} Assume that $v<v'$.
We show that $\det\gamma<0$ by induction on
the number $\ell$ of edges in $\gamma$. 
If $\ell=1$ then $\det\gamma=\det e<0$.
Assume that $\ell>1$ and that the result is true for linear paths shorter than $\gamma$.
Let $v''$ be the vertex such that $e=[v,v'']$ and denote $q(e,v'')$ by  $a$ and $Q(e,v'')$ by  $a'$.
Let $\gamma'$ be the linear path from $v''$ to $v'$.  Let
$$
D = q a - Q a' = \det e < 0 \quad \text{and} \quad D' = q' a' - a Q' = \det \gamma' < 0.
$$
Since $v<v'$, we have $q>0$, $a'>0$ and $Q'>0$ by \ref{p9823p98p2d}\eqref{1823urcpiwd}.
Then
$$
\det\gamma
= qq'-QQ'
= \frac{ q(D'+aQ') - (qa-D)Q' }{ a' }
= \frac{ qD' + Q'D }{ a' } < 0,
$$
completing the proof by induction.

Since $v<v'$, $\hat x_{v,\alpha} > 0$ for every $\alpha \in A'$;
as $A' \neq \emptyset$ and $\det(\gamma)<0$, 
$$
\left| \begin{matrix} q & Q' \\ N_v & N_{v'} \end{matrix} \right| = 
\det(\gamma) \sum_{\alpha \in A'} \hat x_{v,\alpha} < 0.
$$
So assertion~\eqref{61324d1g3r} is proved.
\end{proof}



\begin{remark} \label {i875863urj1dpi}
By \ref{kuwdhr12778}\eqref{61324d1g3r}, 
if $v,v' \in \Veul$ satisfy $v<v'$ and $N_v\le0$, then $N_{v'} < 0$.
It follows that  the set of vertices whose multiplicity is nonnegative  is connected.
\end{remark}

\begin{remark} \label {es5s6xjc17lcy4l}
Let $e=[v,\alpha]$ be a dead end, where $v \in \Veul$ and $\alpha \in \Aeul_0$.
Then 
$$
N_v = q(e,v) N_{\alpha}.
$$
To see this, simply observe that $x_{v,\beta}=q(e,v)x_{\alpha,\beta}$ for all $\beta \in \Aeul \setminus \Aeul_0$.
\end{remark}

\begin{definition} \label {823yi867jcn73}
A {\it dicritical vertex\/} is a vertex $v \in \Veul$ satisfying $N_v=0$.

If $v$ is a dicritical vertex satisfying
\begin{equation} \tag {$ * $}
\setspec{ x \in \Veul \cup \Aeul }{ x > v } \subseteq \Aeul ,
\end{equation}
we define the {\it degree of the dicritical\/} $v$ to be the number of edges $[v,\alpha]$ where $\alpha$ is an arrow decorated with $(1)$.
\end{definition}

\begin{remark} \label {zfp2o39wdw0d}
Let $\Teul$ be an abstract Newton tree at infinity all of whose vertices have nonnegative multiplicity
(one says that $\Teul$ is a ``generic'' Newton tree at infinity, see Def.\ \ref{cn1o29cnds9}).
Then Rem.\ \ref{i875863urj1dpi} implies that each dicritical vertex $v$ of $\Teul$ satisfies condition $(*)$ of Def.\ \ref{823yi867jcn73}.
Consequently, each dicritical vertex of $\Teul$ has a well-defined degree.
\end{remark}

\begin{lemma}  \label {i78321723ye8}
Let $v'$ be a dicritical vertex of an abstract Newton tree at infinity and suppose that
$\setspec{ x \in \Veul \cup \Aeul }{ x > v' } \subseteq \Aeul$. Then the following hold.
\begin{enumerate}

\item \label {7359d81723yruhd}
$q(\epsilon,v')=1$ for every edge $\epsilon$ linking $v'$ to an element of $\Aeul \setminus \Aeul_0$.

\item \label {32d3wsxbei23}
Let $\gamma$ be a linear path from a vertex $v \in \Veul \setminus \{v'\}$ to $v'$.
Then $N_v = -s\det \gamma$ where $s\ge1$ is the degree of the dicritical $v'$.

\end{enumerate}
\end{lemma}

\begin{proof}
Let $\gamma$ be a linear path from a vertex $v \in \Veul \setminus \{v'\}$ to $v'$.
Let the notation ($q,Q,q',Q',A,A'$) be that of \ref{kuwdhr12778} applied to $\gamma$.
Write $A' = \{\alpha_1, \dots, \alpha_s\}$ and $\epsilon_i = [v',\alpha_i]$ for $i = 1, \dots, s$.

Consider the case where there is no dead end incident to $v'$. 
Then there exists $j \in \{1, \dots, s\}$ such that 
$q(\epsilon_j,v')=Q'$ and $q(\epsilon_i,v')=1$ for $i\neq j$.
The definition of $N_{v'}$ gives
$$
0 = N_{v'} = \sum_{ i=1 }^s x_{v',\alpha_i} + \sum_{ \beta \in A} x_{v',\beta}
= q' + (s-1)q'Q' + Q' \sum_{ \beta \in A} \hat x_{v',\beta} ,
$$
so $Q' \mid q'$.  As $\gcd(q',Q')=1$ by \ref{p9823p98p2d}, we get $Q'=1$,
which proves assertion~\eqref{7359d81723yruhd} of the Lemma.
Applying \ref{kuwdhr12778} to $\gamma$ gives
$\left| \begin{matrix} q & 1 \\ N_v & 0 \end{matrix} \right| = 
\det(\gamma) \sum_{i=1}^s \hat x_{v,\alpha_i} = s \det(\gamma)$, proving \eqref{32d3wsxbei23}.

Next, we consider the case where there is a dead end $\epsilon_0$ incident to $v'$.
Then \ref{p9823p98p2d}\eqref{1823urcpiwd}
gives $q( \epsilon_0, v') = Q'$ and $q( \epsilon_i, v') = 1$ for $i=1,\dots,s$ (in particular
assertion~\eqref{7359d81723yruhd} of the Lemma is true).  Moreover,
applying \ref{kuwdhr12778} to $\gamma$ gives
$\left| \begin{matrix} q & Q' \\ N_v & 0 \end{matrix} \right| = 
\det(\gamma) \sum_{i=1}^s \hat x_{v,\alpha_i} = Q' s \det(\gamma)$;
assertion~\eqref{32d3wsxbei23} follows by dividing by $-Q' \neq 0$ both sides.
\end{proof}

\begin{lemma}  \label {8c3rquirhlaf}
Let $\Teul$ be an abstract Newton tree at infinity.
Suppose that $e=[v,\alpha]$ is an edge such that $\alpha$ is an arrow decorated with $(1)$ and $v$ is a vertex such
that $N_v>0$.
Then there exists a unique integer $q'$ such that,
if $\Teul'$ is the tree obtained from $\Teul$ by the following operations:
\begin{itemize}

\item remove the edge $e$ (without removing $v$ and $\alpha$),

\item add a vertex $v'$ and two edges $e_v = [v,v']$ and $e_\alpha = [v',\alpha]$,

\item set $q(e_v, v) = q(e,v)$, $q(e_v, v') = q'$ and $q(e_\alpha, v') = 1$,

\end{itemize}
then $\Teul'$ is an abstract Newton tree at infinity and $v'$ is a dicritical of $\Teul'$.
\end{lemma}

\begin{proof}
Let $Q = Q(e,v)$ and $q=q(e,v)>0$ in $\Teul$. For an arbitrary $q' \in \Integ$, 
$\Teul'$ is an abstract Newton tree at infinity if and only if $qq'-Q<0$.
Indeed, this is the condition $\det(e_v)<0$ (in $\Teul'$) and all other conditions
of definition~\ref{p9823p98p2d} are met.
We also note that if $q'$ is such that $\Teul'$ is an abstract Newton tree at infinity,
then the multiplicity $N_v$ of $v$ is the same in $\Teul$ and $\Teul'$,
and the number $Q(e_v,v)$ (in $\Teul'$) is equal to $Q=Q(e,v)$.

Let $A = (\Aeul \setminus \Aeul_0) \setminus \{ \alpha \}$ in $\Teul$. Then,
calculating in $\Teul$, we have
$N_v = x_{v,\alpha} + \sum_{\beta \in A} x_{v,\beta}$
where $x_{v,\alpha} = Q$ and $q \mid x_{v,\beta}$ for each $\beta \in A$.
So $q \mid (Q-N_v)$.  Let $q' = (Q-N_v)/q \in \Integ$, then $qq'-Q = -N_v < 0$,
so $\Teul'$ is an abstract Newton tree at infinity. Applying
\ref{kuwdhr12778}\eqref{238dhwi912F} to the edge $e_v$ of $\Teul'$ gives
\begin{equation}  \label {093hf923e901}
\left| \begin{matrix} q & 1 \\ N_v & N_{v'} \end{matrix} \right| = 
(qq'-Q) \sum_{\beta \in \{\alpha\}} \hat x_{v,\beta}
= qq'-Q . 
\end{equation}
Since $qq'-Q = -N_v$, this gives $q N_{v'}=0$ and hence $N_{v'}=0$, showing that $q'$ exists.

To prove uniqueness, consider any $q' \in \Integ$ such that 
$\Teul'$ is an abstract Newton tree at infinity and such that $N_{v'}=0$.
Then \eqref{093hf923e901} is still valid and gives $-N_v = qq'-Q$, so 
$q' = (Q-N_v)/q$ is unique.
\end{proof}

\begin{definition}
The {\it number of points at infinity\/} of an abstract Newton tree at infinity
is the valency of the root when there is no dead end attached to the root and the valency minus one otherwise.
\end{definition}

\begin{definition}
We define the {\it multiplicity\/} $M(\Teul)$ of an abstract Newton tree at infinity $\Teul$ by
$ M(\Teul) = -\sum_{v \in \Veul\cup \Aeul_0} N_v(\delta_v-2) $.
\end{definition}

\begin{definition} \label {83724d91872trhy}
Let $\Teul$ be an abstract Newton tree at infinity.
\begin{enumerate}

\item \label {pb2838309udjo}
Suppose that $e=[v,v']$ is a dead end, where $v \in \Veul$, $v' \in \Aeul_0$ and such
that the decoration of $e$ near $v$ is $1$. 
Let $\Teul'$ be the  abstract Newton tree at infinity obtained by deleting $e$ and $v'$ from $\Teul$.
We say that $\Teul'$ is obtained from $\Teul$ {\it by removing a dead end decorated by $1$}.
(The tree $\Teul'$ is indeed an abstract Newton tree at infinity because,
for any $v \in \Veul \setminus \{v_0\}$, part (1)(a) of Definition~\ref{p9823p98p2d} implies that
$\delta_v \ge 2 + n$ where $n \in \{0,1\}$ is the number of dead ends attached to $v$.)

\item \label {4123eefdx64xry8h}
Let $v$ be a vertex different from the root and of valency $2$, let $e=[v,u]$ and $e'=[v,u']$ be the edges 
incident to $v$, and
let $d = q(e,u)$ and $d' = q(e',u')$.
Let $\Teul'$ be the abstract Newton tree at infinity obtained by deleting $v$, $e$ and $e'$ from $\Teul$,
and adding the edge $e''=[u,u']$, where the decorations of $e''$ near $u$ and $u'$ are $d$ and $d'$ respectively.
We say that $\Teul'$ is obtained from $\Teul$ {\it by removing a vertex of valency $2$.}
(Note that $\det e'' < 0$ by \ref{kuwdhr12778}\eqref{61324d1g3r}, so $\Teul'$ is indeed an abstract Newton tree at infinity.)

\item Note that, in (1) and (2),
we have $\Veul' \cup \Aeul' \subset \Veul \cup \Aeul$, where 
$\Veul, \Aeul, \Veul', \Aeul'$ are respectively the sets of 
vertices of $\Teul$, arrows of $\Teul$, vertices of $\Teul'$ and arrows of $\Teul'$.

\item Given abstract Newton trees at infinity $\Teul$ and $\Teul'$, we write $\Teul \ge \Teul'$ to
indicate that either $\Teul=\Teul'$ or there exists a finite sequence $\Teul_0, \dots, \Teul_n$
of abstract Newton trees at infinity satisfying $\Teul_0=\Teul$, $\Teul_n=\Teul'$ and,
for each $i \in \{0,\dots, n-1\}$,
$\Teul_{i+1}$ is obtained from $\Teul_{i}$ by removing a dead end decorated by $1$ 
or a vertex of valency $2$.
If $\Teul' \le \Teul$ then $\Veul' \cup \Aeul' \subseteq \Veul \cup \Aeul$ and, for each $v \in \Veul' \cup \Aeul_0'$,
the multiplicities of $v$ in $\Teul$ and $\Teul'$ are equal.
Note that $\le$ is a partial order on the set of 
abstract Newton trees at infinity.

\item Whenever we say that  abstract Newton trees at infinity are equivalent,
we mean that they are so with respect to the equivalence relation on the set of abstract Newton trees at infinity
which is generated by the relation $\le$.

\end{enumerate}
\end{definition}

\begin{lemma} \label {0923923ujds0}
Equivalent abstract  Newton trees at infinity have the same number of points at infinity and the same multiplicity.
\end{lemma}

\begin{proof}
Consider the situation of part (1) of Definition~\ref{83724d91872trhy}.
The contribution of $v$ and $v'$ to $M(\Teul)$ is $-N_v(\delta_v-2) + N_{v'}$,
while the contribution of $v$ to $M(\Teul')$ is $-N_v(\delta_v-1-2)$; as $N_v=N_{v'}$ (in $\Teul$)
by Remark~\ref{es5s6xjc17lcy4l}, it follows that $M(\Teul)=M(\Teul')$ in that case.
The rest of the argument is quite clear and the details are left to the reader.
\end{proof}

\begin{definition} \label {cn1o29cnds9}
A {\it generic Newton tree at infinity\/} is an abstract Newton tree at infinity all of whose vertices 
have  nonnegative multiplicity.
A {\it complete Newton tree at infinity\/} is a generic Newton tree at infinity
in which each arrow decorated with $(1)$ is adjacent to a dicritical vertex.
\end{definition}

\begin{remark}
Let $\Teul$ be an abstract Newton tree at infinity.
By \ref{i875863urj1dpi}, 
there is at most one dicritical on any path from the root to an arrow.
If $\Teul$ is complete, there is exactly one dicritical on any path from the root to an arrow decorated by $(1)$.
Also note that if $\Teul$ is complete then its number of points at infinity is equal to the number
of edges $[v_0,v]$ with $v \in \Veul$.
\end{remark}

\begin{definition} \label {e5e6e7w8e9wee9}
A Newton tree at infinity $\Teul$ is {\it minimally complete\/} if it satisfies:
\begin{itemize}
\item[(i)] $\Teul$ is complete;

\item[(ii)] if $v$ is a dicritical then there is a dead end incident to $v$;

\item[(iii)] if a dead end decorated by $1$ is incident to a vertex $v$, then $v$ is a dicritical;

\item[(iv)] every element of $\Veul \setminus \{v_0\}$ has valency different from $2$.

\end{itemize}
\end{definition}

\begin{lemma}
Let $\Ceul$ be the equivalence class of a generic Newton tree at infinity.
Then $\Ceul$ contains exactly one minimally complete  Newton tree at infinity.
\end{lemma}

\begin{proof}
Let $\Teul$ be a generic Newton tree at infinity such that $\Teul \in \Ceul$.
Let $\Teul'$ be obtained from $\Teul$ by first removing all
dead ends decorated by $1$ (in the sense of \ref{83724d91872trhy}\eqref{pb2838309udjo})
and then removing all vertices of valency $2$ (in the sense of \ref{83724d91872trhy}\eqref{4123eefdx64xry8h}; note that 
the root cannot be removed). Then $\Teul'$ is a generic Newton tree at infinity such that $\Teul' \le \Teul$.
Applying \ref{8c3rquirhlaf} successively  to each 
edge $e=[v,\alpha]$ of $\Teul'$ such that $\alpha$ is an arrow decorated with $(1)$
and $v$ a vertex with $N_v>0$ produces a complete Newton tree $\Teul''$ satisfying
$\Teul'' \ge \Teul'$. 
For each dicritical $v'$ of $\Teul''$ such that no dead end is incident to it,
we add a dead end decorated by $1$ incident to $v'$; by \ref{i78321723ye8}\eqref{7359d81723yruhd},
this addition does not violate \ref{p9823p98p2d}\eqref{1823urcpiwd}, so the resulting tree $\Teul'''$ is
an abstract Newton tree at infinity
(the fact that $v'$ satisfies the hypothesis of \ref{i78321723ye8} follows from \ref{i875863urj1dpi}).
Moreover, $\Teul'''$ is minimally complete and satisfies $\Teul''' \ge \Teul''$.
This proves that there exists an element of $\Ceul$ which is minimally complete; uniqueness is left to the reader.
\end{proof}

\begin{notations}
Let $\Teul$ be an abstract Newton tree at infinity.
Given $v \in \Veul\cup \Aeul_0$, define $C(v) = -N_v(\delta_v-2)$;
given a subset $W$ of $\Veul\cup \Aeul_0$, define $C(W) = \sum_{v \in W} C(v)$.
Given $v \in \Veul$, let
$$
W_v = \{v\} \cup \setspec{ \alpha \in \Aeul_0 }{ \text{$\alpha$ is linked to $v$ by an edge} }
\text{ and $\bar C(v) = C(W_v)$;}
$$
given a subset $W$ of $\Veul$, define $\bar C(W) = \sum_{v \in W} \bar C(v)$.
Note that $M(\Teul) = \bar C(\Veul)$. 

Let $\Neul = \setspec{ v \in \Veul }{ \text{$v$ is not dicritical} }$.
Observe that $\bar C(v)=0$ whenever $v$ is a dicritical vertex (see Remark~\ref{es5s6xjc17lcy4l});
so $\bar C(\Neul)= \bar C(\Veul) = M(\Teul)$.
For each $v \in \Neul$, write 
\begin{align*}
r_v &= -1 + | \setspec{v' \in \Veul}{\text{$[v,v']$ is an edge}} | \\
\Delta(v) &= 1-r_v - \bar C(v) .
\end{align*}
For any subset $W$ of $\Neul$, let $\Delta(W) = \sum_{v \in W} \Delta(v)$.

Finally, for each $v \in \Veul$
define $a_v=1$ if there is no dead end incident to $v$,
and $a_v = q(e,v)$ if $e$ is a dead end incident to $v$.
Note that $a_v \mid N_v$, by \ref{es5s6xjc17lcy4l}.
\end{notations}

Note that $v_0 \in \Neul$, in any abstract Newton tree at infinity.
Also observe that if $\Teul$ is minimally complete, then no edge links an arrow to the root and $\delta_{v_0}$
is equal to the number of points at infinity.

\begin{lemma} \label {f1z2fzeddsz1240}
Let $\Teul$ a minimally complete Newton tree at infinity with at least two points at infinity.
Then the following hold.
\begin{enumerate}

\item \label {092f39rdjwiedjp293}
For each $v \in \Neul$, $r_v\ge1$,
$\Delta(v) = (r_v-1)(N_v-1) + N_v\big( 1-\frac1{a_v} \big)$ and $\Delta(v) \ge 0$.

\item \label {9f23p09jwoiej} For each $v \in \Neul \setminus\{ v_0 \}$, $\Delta(v)=0$ if and only if  $N_v=1$ and $W_v=\{v\}$.

\item \label {982ye127r4g} $\Delta(v_0)=0$ if and only if $\delta_{v_0}=2$.

\end{enumerate}
\end{lemma}

\begin{proof}
Let $v \in \Neul$.
We prove:
\begin{equation} \label {29390fjcei}
\textstyle \Delta(v) = (N_v-1)(r_v-1) + N_v\big( 1-\frac1{a_v} \big).
\end{equation}
First note that,
since $\Teul$ is complete and $v$ is not a dicritical, 
no element of $\Aeul \setminus \Aeul_0$ is linked to $v$ by an edge;
so $\delta_v$ is equal to $r_v+1$ plus the number ($0$ or $1$) of dead ends incident to $v$.
If $v$ is not the root then $\delta_v\ge3$, so $r_v\ge1$;
if $v$ is the root then $\delta_v \ge 2$ and no arrow is linked to $v$ by an edge, so again $r_v\ge1$.

Consider the case where $W_v=\{v\}$. Then $\delta_v = r_v+1$ and $a_v=1$, so
\begin{gather*}
\bar C(v) = C(W_v) = -N_v(\delta_v-2) = -N_v(r_v-1),
\end{gather*}
so $\Delta(v)  = 1-r_v - \bar C(v) = (N_v-1)(r_v-1)$ and hence \eqref{29390fjcei} holds, since $a_v=1$.

If $W_v \neq \{v\}$ then, for each $\alpha \in W_v \setminus \{v\}$, $[v,\alpha]$ is a dead end incident to $v$;
so (by \ref{p9823p98p2d}) $W_v = \{v,\alpha\}$ for some $\alpha \in \Aeul_0$.
We have $\delta_v = r_v+2$ and, by \ref{es5s6xjc17lcy4l}, $N_v = a_v N_{\alpha}$. We get
\begin{gather*}
\bar C(v) = C(W_v) = -N_v(\delta_v-2) - N_\alpha(\delta_\alpha-2) = -N_v r_v + N_\alpha, \\ 
\textstyle \Delta(v) = 1-r_v - \bar C(v) = 1-r_v + N_v r_v - N_\alpha = (N_v-1)(r_v-1) + N_v\big( 1-\frac1{a_v} \big),
\end{gather*}
which proves \eqref{29390fjcei}.

As $r_v, N_v, a_v \ge 1$, we have 
$(N_v-1)(r_v-1) \ge0$ and $N_v\big( 1-\frac1{a_v} \big) \ge 0$, so $\Delta(v) \ge 0$
(the fact that $r_v\ge1$ follows from the assumptions that $\Teul$ has at least two points at infinity and is complete);
moreover, $\Delta(v)=0$ is equivalent
to $(N_v-1)(r_v-1) =0$ and $N_v\big( 1-\frac1{a_v} \big) =0$,
which is equivalent to  $a_v=1$ and either $N_v=1$ or $r_v=1$.
Assertions \eqref{9f23p09jwoiej} and \eqref{982ye127r4g} follow.
\end{proof}

\begin{lemma} \label {823yc27dfudhi}
Let $\Teul$ a minimally complete Newton tree at infinity with at least two points at infinity.
Consider an ordered pair $e=(v_*,v_1) \in \Veul\times\Veul$ such that $[v_*,v_1]$ is an edge.
Let $\Veul_e = \setspec{ x \in \Veul }{ \text{the path from $v_1$ to $x$ does not contain $v_*$} }$,
$$
\Deul_e = \setspec{ v \in \Veul_e }{ \text{$v$ is a dicritical} },\ 
\Neul_e = \setspec{ v \in \Veul_e }{ \text{$v$ is not a dicritical} }
$$
and $d_e = |\Deul_e|$.
Then $\Delta(\Neul_e) = 1 - d_e -\bar C(\Neul_e)$ and $\Delta(\Neul_e) \ge 0$,
where  $\Delta(\Neul_e) = 0$ if and only if $\Delta(v)=0$ for all $v \in \Neul_e$.
\end{lemma}

\begin{proof}
For each $i\ge1$,
we denote by $\Veul_i$ the set of vertices in $\Veul_e$ which are at distance $i$ from $v_*$;
we also set
$$
\Deul_i = \setspec{ v \in \Veul_i }{ \text{$v$ is a dicritical} },
\quad
\Neul_i = \setspec{ v \in \Veul_i }{ \text{$v$ is not a dicritical} },
$$
$d_i = |\Deul_i|$ and $n_i = |\Neul_i|$.
Observe that $\sum_{v \in \Neul_i} r_v = |\Veul_{i+1}| = n_{i+1}+d_{i+1}$
and that $\Delta(v) = 1-r_v - \bar C(v)$ for each $v \in \Neul_i$.
Then
\begin{multline*}
\Delta(\Neul_i)
= \sum_{v \in \Neul_i} \Delta(v)
= \sum_{v \in \Neul_i} (1-r_v-\bar C(v))
= n_i - \sum_{v \in \Neul_i} r_v - \bar C(\Neul_i)
= n_i - n_{i+1} - d_{i+1} - \bar C(\Neul_i),
\end{multline*}
so
\begin{multline*}
\Delta( \Neul_e )
= \sum_{i=1}^\infty \Delta(\Neul_i)
= \sum_{i=1}^\infty (n_i - n_{i+1} - d_{i+1} - \bar C(\Neul_i))
= n_1 - \sum_{i=1}^\infty d_{i+1} - \bar C(\Neul_e) \\
= n_1 + d_1 - d_e - \bar C(\Neul_e)
= 1 - d_e - \bar C(\Neul_e).
\end{multline*}
By Lemma~\ref{f1z2fzeddsz1240}, $\Delta(v) \ge 0$ for all $v \in \Neul_e$.
So $\Delta( \Neul_e ) \ge0$ where equality holds if and only if 
$\Delta(v) = 0$ for all $v \in \Neul_e$.
\end{proof}

\begin{theorem} \label {918235071yrsj2dhry}
Let $\Teul$ be a minimally complete Newton tree at infinity with at least two points at infinity.
Let $u_1, \dots, u_s$ be the dicriticals of $\Teul$, where $u_i$ is of degree $d_i$.
Then
\begin{enumerate}

\item $M(\Teul) + \Delta(\Neul) = 2-s$

\item $M(\Teul) = 2 - \sum_{i=1}^s d_i$ if and only if $\Delta(\Neul) = \sum_{i=1}^s (d_i-1)$.

\end{enumerate}
\end{theorem}

\begin{proof}
Choose an ordered pair $e=(v_*,v_1) \in \Veul\times\Veul$ such that $[v_*,v_1]$ is an edge and $v_*$ 
is a dicritical.
With notation as in Lemma~\ref{823yc27dfudhi}, we have $\Neul_e=\Neul$ and $d_e = s-1$. So the Lemma gives
$$
\Delta(\Neul) 
= \Delta(\Neul_e) 
= 1-d_e - \bar C(\Neul_e)
= 2-s - \bar C(\Neul)
= 2-s - \bar C(\Veul) = 2-s- M(\Teul),
$$
which proves the first assertion.  The second assertion follows.
\end{proof}

\begin{corollary} \label {8712e7g1duddddwcbkdu3}
Let $\Teul$ be a minimally complete Newton tree at infinity with $s$ dicriticals and at least two points at infinity.
Then $M(\Teul) \le 2-s$ and equality holds if and only if 
there are exactly two points at infinity,
every dead end is incident to a dicritical and every vertex 
which is not dicritical and which is different from the root has multiplicity~$1$. 
\end{corollary}

\begin{proof}
We have $M(\Teul) = 2-s - \Delta(\Neul)$ 
by the first part of Theorem~\ref{918235071yrsj2dhry}
and $\Delta(\Neul) \ge 0$ by \ref{f1z2fzeddsz1240}\eqref{092f39rdjwiedjp293},
so $M(\Teul) \le 2-s$;
$M(\Teul) = 2-s$ if and only if  $\Delta(\Neul) = 0$,
so the last assertion follows from \ref{f1z2fzeddsz1240}.
\end{proof}

\begin{lemma} \label {8123oiuwhqd8d2}
Consider an edge $[v,v']$ in a minimally complete Newton tree at infinity,
where $v$ is a vertex of multiplicity $1$, $v'$ is a vertex, and $v<v'$.
Then $v'$ is a dicritical of degree $1$ and the edge determinant of $[v,v']$ is $-1$.
\end{lemma}

\begin{proof}
Let $e=[v,v']$, $q = q(e,v)$, $Q'=Q(e,v')$, and
$$
A' = \setspec{ \alpha \in \Aeul\setminus\Aeul_0 }{ \text{the path from $v$ to $\alpha$ contains $v'$} }.
$$
There exists an edge $\epsilon = [v',v_1]$ such that $v'<v_1$ and $q(\epsilon,v')=Q'$.
Let $A_1' = \setspec{ \alpha \in A' }{ \alpha \ge v_1 }$ and $A_2' = A' \setminus A_1'$.
Then $A_2' \neq \emptyset$, since  $\delta_{v'} \ge3$.
Applying \ref{kuwdhr12778}\eqref{238dhwi912F} to the edge $e=[v,v']$ gives
$$
\left| \begin{matrix} q & Q' \\ 1 & N_{v'} \end{matrix} \right|
= \det(e) \sum_{\alpha \in A'} \hat x_{v,\alpha} = \det(e) ( Q' C + c ),
$$
where $c=\sum_{\alpha \in A_1'} \hat x_{v,\alpha}$ and
$C= \frac1{Q'}\sum_{\alpha \in A_2'} \hat x_{v,\alpha} \in \Nat \setminus \{0\}$. Then
$$
qN_{v'}-Q' = \det(e) (Q'C+c),
$$
so
\begin{equation} \label {89351873yrh}
Q'( -\det(e) C -1 )  =  -q N_{v'} + \det(e) c .
\end{equation}
Since $Q'( -\det(e) C -1 ) \ge 0$ and   $-q N_{v'} + \det(e) c \le 0$, both sides of \eqref{89351873yrh}
are equal to $0$.
It follows that $N_{v'}=0$, $C=1$, $\det(e)=-1$ and $c=0$.
Since $|A_1'| \le c = 0$ and  $|A_2'| \le C = 1$,  we obtain that $\epsilon$ is a dead end and that $v'$ is a dicritical
of degree $1$.
\end{proof}

\section{Newton trees at infinity with maximum multiplicity}

\begin{definition}
Let $\Teul$ be an abstract Newton tree at infinity.
The {\it shadow\/} of $\Teul$, denoted $S(\Teul)$, is the decorated rooted tree obtained from $\Teul$
by erasing the decorations of the edges and decorating each vertex $v \neq v_0$ with its multiplicity $(N_v)$.
(It is understood that $S(\Teul)$ has the same underlying rooted tree as $\Teul$ and that the arrows
have the same decorations $(0)$ or $(1)$ in $S(\Teul)$ and in $\Teul$.)
\end{definition}

\begin{definition}
Let $\Teul$ be a minimally complete Newton tree at infinity with at least two points at infinity.
For each dicritical $u$ of $\Teul$, there exists a unique $v \in \Veul \setminus \{u\}$ such 
that $\delta_v>2$ and the path from $u$ to $v$ is linear.
We call $v$ the {\it companion\/} of $u$.
(Remark: if $[u,v]$ is not an edge, then the path from $u$ to $v$ is $[u,v_0,v]$.
This follows from part (iv) of Definition \ref{e5e6e7w8e9wee9}.)
\end{definition}

\begin{lemma} \label {o9f83e9udos9}
Let $\Teul$ be a minimally complete Newton tree at infinity with at least two points at infinity.
Then TFAE:
\begin{enumerate}

\item Every vertex that is a companion of a dicritical is a dicritical;

\item there exists a dicritical of $\Teul$ whose companion is a dicritical.

\end{enumerate}
If we moreover assume that the gcd of the degrees of the dicriticals of $\Teul$ is $1$,
then the above conditions are equivalent to:
\begin{enumerate}
\addtocounter{enumi}{2}

\item the shadow of $\Teul$ is the one shown in Figure~\ref{78r273wegadskjh278}.

\end{enumerate}
\end{lemma}

\begin{proof}
Assume that (1) holds.
Let $u$ be a dicritical of $\Teul$ and let $v$ be its companion.
Then $v$ is a dicritical, so (2) holds.

Assume that (2) holds and let $u,v$ be dicriticals where $v$ is the companion of $u$.
As two dicriticals cannot be joined by an edge (cf.\ Remark~\ref{i875863urj1dpi}), the path from $u$ to $v$ is $[u,v_0,v]$.
Since this is a linear path, we get $\delta_{v_0}=2$, so $\Veul = \{u,v_0,v\}$
and it follows that (1) holds.

Assume that the gcd of the degrees of the dicriticals of $\Teul$ is $1$.

Suppose that (2) holds and let the notation be as in the proof that (2) implies (1).
Then, to prove (3), there only remains to show that $d_u=1=d_v$, where 
$d_u$ and $d_v$ are the degrees of the dicriticals $u$ and $v$.  We have:
$$
0 = \frac{N_{v}}{a_{v}} = q([v_0,v], v) d_v + a_{u} d_u,
\ \ 
0 = \frac{N_{u}}{a_{u}} = q([v_0,u], u) d_u + a_{v} d_v.
$$
As $\gcd(d_u,d_v)=1$, it follows that $d_u \mid \gcd( q([v_0,v], v), a_{v} )$,
so $d_u=1$ since
$$
\gcd( q([v_0,v], v),  a_{v} ) = 1;
$$
similarly, $d_v=1$, so (3) holds. It is clear that (3) implies (1).
\end{proof}

\begin{theorem}  \label {8c3r8273d287ed2uq98}
Let $\Teul$ be a minimally complete Newton tree at infinity with $d$ dicriticals,
such that the gcd of the degrees of its dicriticals is $1$.
Suppose that $\Teul$ has at least two points at infinity and has multiplicity $2-d$.
Then the shadow of $\Teul$ is one of Figures \ref{78r273wegadskjh278},
\ref{q723y87c23redhsi} and~\ref{273r872qhwudsaku32}, where $r\ge2$ is arbitrary in Figures
\ref{q723y87c23redhsi} and~\ref{273r872qhwudsaku32}.
Conversely, every shadow represented in these figures is indeed the shadow of 
a minimally complete Newton tree at infinity satisfying the above hypotheses.
\end{theorem}

\begin{smallremark}
{\setlength{\unitlength}{1mm}
In the pictures below, dicritical vertices are represented by 
``\begin{picture}(2,1)(-1,-.5) \put(0,0){\circle*{1.5}} \end{picture}''
and vertices that are not dicritical are represented by 
``\begin{picture}(2,1)(-1,-.5) \put(0,0){\circle{1.5}} \end{picture}''.}
\end{smallremark}

\begin{figure}[ht]
\begin{center}
%
\setlength{\unitlength}{2763sp}%
\begingroup\makeatletter\ifx\SetFigFont\undefined%
\gdef\SetFigFont#1#2#3#4#5{%
  \reset@font\fontsize{#1}{#2pt}%
  \fontfamily{#3}\fontseries{#4}\fontshape{#5}%
  \selectfont}%
\fi\endgroup%
\begin{picture}(4224,1636)(364,-2825)
{\color[rgb]{0,0,0}\thinlines
\put(1201,-1561){\circle*{150}}
}%
{\color[rgb]{0,0,0}\put(3601,-1561){\circle*{150}}
}%
{\color[rgb]{0,0,0}\put(1126,-1561){\vector(-1, 0){750}}
}%
{\color[rgb]{0,0,0}\put(1276,-1561){\line( 1, 0){1050}}
}%
{\color[rgb]{0,0,0}\put(2476,-1561){\line( 1, 0){1050}}
}%
{\color[rgb]{0,0,0}\put(3676,-1561){\vector( 1, 0){900}}
}%
{\color[rgb]{0,0,0}\put(3601,-1636){\vector( 0,-1){900}}
}%
{\color[rgb]{0,0,0}\put(1201,-1636){\vector( 0,-1){900}}
}%
\put(1051,-1336){\makebox(0,0)[lb]{\smash{{\SetFigFont{8}{9.6}{\rmdefault}{\mddefault}{\updefault}{\color[rgb]{0,0,0}(0)}%
}}}}
\put(3451,-1336){\makebox(0,0)[lb]{\smash{{\SetFigFont{8}{9.6}{\rmdefault}{\mddefault}{\updefault}{\color[rgb]{0,0,0}(0)}%
}}}}
\put(1051,-2761){\makebox(0,0)[lb]{\smash{{\SetFigFont{8}{9.6}{\rmdefault}{\mddefault}{\updefault}{\color[rgb]{0,0,0}(0)}%
}}}}
\put(3526,-2761){\makebox(0,0)[lb]{\smash{{\SetFigFont{8}{9.6}{\rmdefault}{\mddefault}{\updefault}{\color[rgb]{0,0,0}(0)} }}}}
\put(901,-1786){\makebox(0,0)[lb]{\smash{{\SetFigFont{8}{9.6}{\rmdefault}{\mddefault}{\updefault}{\color[rgb]{0,0,0}$v_1$}%
}}}}
\put(3301,-1786){\makebox(0,0)[lb]{\smash{{\SetFigFont{8}{9.6}{\rmdefault}{\mddefault}{\updefault}{\color[rgb]{0,0,0}$v_1'$} }}}}
\put(2326,-1336){\makebox(0,0)[lb]{\smash{{\SetFigFont{8}{9.6}{\rmdefault}{\mddefault}{\updefault}{\color[rgb]{0,0,0}$v_0$} }}}}
{\color[rgb]{0,0,0}\put(2401,-1561){\circle{150}}
}%
\end{picture}%
%
\caption{One dicritical on each side of the root.}
\label {78r273wegadskjh278}
\end{center}
\end{figure}
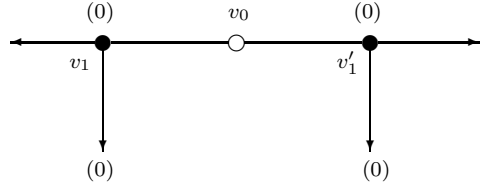  
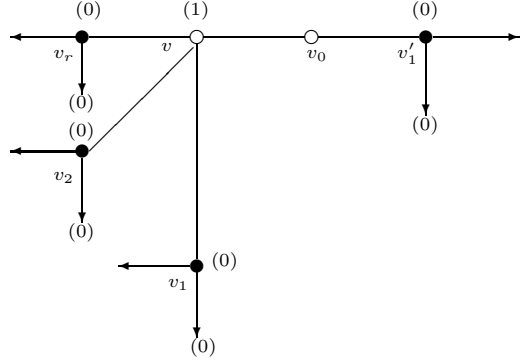
\begin{figure}[ht]
\begin{center}
\setlength{\unitlength}{2368sp}%
\begingroup\makeatletter\ifx\SetFigFont\undefined%
\gdef\SetFigFont#1#2#3#4#5{%
  \reset@font\fontsize{#1}{#2pt}%
  \fontfamily{#3}\fontseries{#4}\fontshape{#5}%
  \selectfont}%
\fi\endgroup%
\begin{picture}(5349,3736)(439,-6725)
{\color[rgb]{0,0,0}\thinlines \put(1201,-3361){\circle*{150}} }%
{\color[rgb]{0,0,0}\put(1201,-4561){\circle*{150}}
}%
{\color[rgb]{0,0,0}\put(2401,-5761){\circle*{150}}
}%
{\color[rgb]{0,0,0}\put(4801,-3361){\circle*{150}}
}%
{\color[rgb]{0,0,0}\put(3601,-3361){\circle{150}}
}%
{\color[rgb]{0,0,0}\put(1126,-3361){\vector(-1, 0){675}}
}%
{\color[rgb]{0,0,0}\put(1126,-4561){\vector(-1, 0){675}}
}%
{\color[rgb]{0,0,0}\put(1201,-3436){\vector( 0,-1){525}}
}%
{\color[rgb]{0,0,0}\put(1201,-4636){\vector( 0,-1){675}}
}%
{\color[rgb]{0,0,0}\put(1276,-3361){\line( 1, 0){1050}}
}%
{\color[rgb]{0,0,0}\put(1276,-4561){\line( 1, 1){1087.500}}
}%
{\color[rgb]{0,0,0}\put(2401,-3436){\line( 0,-1){2250}}
}%
{\color[rgb]{0,0,0}\put(2326,-5761){\vector(-1, 0){750}}
}%
{\color[rgb]{0,0,0}\put(2401,-5836){\vector( 0,-1){675}}
}%
{\color[rgb]{0,0,0}\put(4801,-3436){\vector( 0,-1){750}}
}%
{\color[rgb]{0,0,0}\put(4876,-3361){\vector( 1, 0){900}}
}%
{\color[rgb]{0,0,0}\put(2476,-3361){\line( 1, 0){1050}}
}%
{\color[rgb]{0,0,0}\put(3676,-3361){\line( 1, 0){1050}}
}%
{\color[rgb]{0,0,0}\put(1201,-3436){\line( 0,-1){375}}
}%

\put(1051,-4111){\makebox(0,0)[lb]{\smash{{\SetFigFont{7}{8.4}{\rmdefault}{\mddefault}{\updefault}{\color[rgb]{0,0,0}(0)} }}}}
\put(2326,-6661){\makebox(0,0)[lb]{\smash{{\SetFigFont{7}{8.4}{\rmdefault}{\mddefault}{\updefault}{\color[rgb]{0,0,0}(0)} }}}}
\put(4651,-4336){\makebox(0,0)[lb]{\smash{{\SetFigFont{7}{8.4}{\rmdefault}{\mddefault}{\updefault}{\color[rgb]{0,0,0}(0)} }}}}
\put(1126,-3136){\makebox(0,0)[lb]{\smash{{\SetFigFont{7}{8.4}{\rmdefault}{\mddefault}{\updefault}{\color[rgb]{0,0,0}(0)} }}}}
\put(1050,-4400){\makebox(0,0)[lb]{\smash{{\SetFigFont{7}{8.4}{\rmdefault}{\mddefault}{\updefault}{\color[rgb]{0,0,0}(0)} }}}}
\put(2251,-3136){\makebox(0,0)[lb]{\smash{{\SetFigFont{7}{8.4}{\rmdefault}{\mddefault}{\updefault}{\color[rgb]{0,0,0}(1)} }}}}
\put(2551,-5761){\makebox(0,0)[lb]{\smash{{\SetFigFont{7}{8.4}{\rmdefault}{\mddefault}{\updefault}{\color[rgb]{0,0,0}(0)} }}}}
\put(4651,-3136){\makebox(0,0)[lb]{\smash{{\SetFigFont{7}{8.4}{\rmdefault}{\mddefault}{\updefault}{\color[rgb]{0,0,0}(0)} }}}}
\put(1051,-5461){\makebox(0,0)[lb]{\smash{{\SetFigFont{7}{8.4}{\rmdefault}{\mddefault}{\updefault}{\color[rgb]{0,0,0}(0)} }}}}

\put(3550,-3586){\makebox(0,0)[lb]{\smash{{\SetFigFont{7}{8.4}{\rmdefault}{\mddefault}{\updefault}{\color[rgb]{0,0,0}$v_0$} }}}}

\put(2026,-3511){\makebox(0,0)[lb]{\smash{{\SetFigFont{7}{8.4}{\rmdefault}{\mddefault}{\updefault}{\color[rgb]{0,0,0}$v$} }}}}

\put(2101,-5986){\makebox(0,0)[lb]{\smash{{\SetFigFont{7}{8.4}{\rmdefault}{\mddefault}{\updefault}{\color[rgb]{0,0,0}$v_1$} }}}}

\put(901,-4861){\makebox(0,0)[lb]{\smash{{\SetFigFont{7}{8.4}{\rmdefault}{\mddefault}{\updefault}{\color[rgb]{0,0,0}$v_2$} }}}}

\put(901,-3586){\makebox(0,0)[lb]{\smash{{\SetFigFont{7}{8.4}{\rmdefault}{\mddefault}{\updefault}{\color[rgb]{0,0,0}$v_r$} }}}}

\put(4501,-3586){\makebox(0,0)[lb]{\smash{{\SetFigFont{7}{8.4}{\rmdefault}{\mddefault}{\updefault}{\color[rgb]{0,0,0}$v_1'$} }}}}

{\color[rgb]{0,0,0}\put(2401,-3361){\circle{150}} }
\end{picture}%
\caption{One dicritical on one side of the root, and $r\ge2$ dicriticals on the other side.}
\label {q723y87c23redhsi}
\end{center}
\end{figure} 
\begin{figure}[ht]
\begin{center}
\setlength{\unitlength}{.8mm}%
\begin{picture}(100,52)(-50,-48.5)

\put(-40,0){\circle*{1.7}}
\put(40,0){\circle*{1.7}}
\put(-20,0){\circle{1.7}}
\put(0,0){\circle{1.7}}
\put(20,0){\circle{1.7}}
\put(-20,-35){\circle*{1.7}}
\put(20,-35){\circle*{1.7}}

\put(-40,0){\vector(-1,0){10}}
\put(-40,0){\vector(0,-1){10}}
\put(-20,-35){\vector(-1,0){10}}
\put(-20,-35){\vector(0,-1){10}}

\put(40,0){\vector(1,0){10}}
\put(40,0){\vector(0,-1){10}}
\put(20,-35){\vector(1,0){10}}
\put(20,-35){\vector(0,-1){10}}

\put(-19.1,0){\line(1,0){18.2}}
\put(-39.1,0){\line(1,0){18.2}}
\put(0.9,0){\line(1,0){18.2}}
\put(20.9,0){\line(1,0){18.2}}

\put(20,-.9){\line(0,-1){33.2}}
\put(-20,-.9){\line(0,-1){33.2}}

\put(-40,-20){\circle*{1.7}}
\put(-40,-20){\vector(-1,0){10}}
\put(-40,-20){\vector(0,-1){10}}
\put(-40,-20){\line(1,1){19.3}}
\put(-40,-19){\makebox(0,0)[rb]{\tiny $(0)$}}
\put(-40,-31){\makebox(0,0)[t]{\tiny $(0)$}}
\put(-41,-21){\makebox(0,0)[rt]{\tiny $v_2$}}

\put(-40,1.5){\makebox(0,0)[b]{\tiny $(0)$}}
\put(-41,-10){\makebox(0,0)[r]{\tiny $(0)$}}
\put(-20,1.5){\makebox(0,0)[b]{\tiny $(1)$}}
\put(20,1.5){\makebox(0,0)[b]{\tiny $(1)$}}
\put(40,1.5){\makebox(0,0)[b]{\tiny $(0)$}}
\put(-20,-46){\makebox(0,0)[t]{\tiny $(0)$}}
\put(-19,-36){\makebox(0,0)[lb]{\tiny $(0)$}}
\put(-21,-36){\makebox(0,0)[rt]{\tiny $v_1$}}
\put(-41,-1){\makebox(0,0)[rt]{\tiny $v_r$}}
\put(21,-36){\makebox(0,0)[lt]{\tiny $v_1'$}}
\put(41,-1){\makebox(0,0)[lt]{\tiny $v_2'$}}

\put(41,-10){\makebox(0,0)[l]{\tiny $(0)$}}
\put(19,-36){\makebox(0,0)[rb]{\tiny $(0)$}}
\put(20,-46){\makebox(0,0)[t]{\tiny $(0)$}}

\put(0,-1.5){\makebox(0,0)[t]{\tiny $v_0$}}
\put(-19,-1){\makebox(0,0)[lt]{\tiny $v$}}
\put(19,-1){\makebox(0,0)[rt]{\tiny $v'$}}

\end{picture}
\caption{Two dicriticals on one side of the root, and $r\ge2$ dicriticals on the other side.}
\label {273r872qhwudsaku32}
\end{center}
\end{figure}
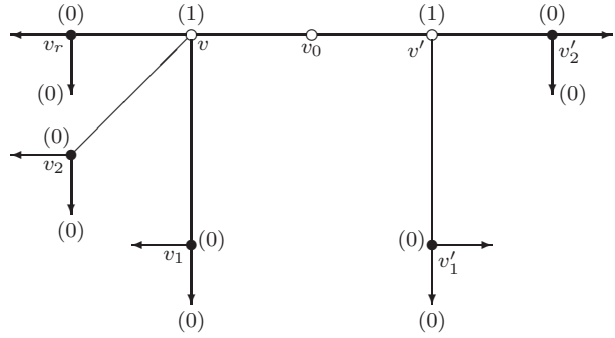  

\begin{proof}
Let $\Teul$ be a minimally complete Newton tree at infinity satisfying the above hypotheses.
By \ref{8712e7g1duddddwcbkdu3} and \ref{8123oiuwhqd8d2}, $S(\Teul)$ has the following properties:
\begin{itemize}

\item[(i)] the root has valency $2$;

\item[(ii)] there exists a dead end incident to a vertex $v$ if and only if $v$ is a dicritical (i.e., is decorated by~$(0)$);

\item[(iii)] all vertices other than $v_0$ are decorated by $(0)$ or $(1)$;

\item[(iv)] if $[v,v']$ is an edge in $S(\Teul)$ where $v_0 \neq v<v'$ are vertices and $v$ is decorated by $(1)$,
then $v'$ is a dicritical;

\item[(v)] every dicritical has degree $1$.

\end{itemize}

Note that claim (iv) follows from Lemma~\ref{8123oiuwhqd8d2}.  Let us justify claim (v).
Let $v'$ be a dicritical vertex.
If there exists a vertex of multiplicity $1$ then there exists a linear path from $v'$ to some vertex $v$
such that $N_v=1$; then the degree of the dicritical $v'$ divides $N_v$
(by \ref{i78321723ye8}\eqref{32d3wsxbei23}) and hence is equal to $1$.
Consider the case where there is no vertex of multiplicity $1$ in $\Teul$.
Then $\Veul = \{ v_0, v, v' \}$ where $v$ and $v'$ are dicriticals.
Moreover, $v$ is a companion of $v'$, so condition~(2) of Lemma~\ref{o9f83e9udos9} is satisfied;
it follows that the shadow of $\Teul$ is the one shown in Figure~\ref{78r273wegadskjh278},
in which case claim~(v) is true.  So claim~(v) is true in all cases.

If there is no vertex decorated by $(1)$ then,
as we saw in the preceding paragraph,
$S(\Teul)$ is as in Figure \ref{78r273wegadskjh278};
in this case we have exactly two dicriticals.
If there is one vertex decorated with $(1)$ then $S(\Teul)$ is as in Figure \ref{q723y87c23redhsi};
in this case the number of dicriticals is not bounded.

If there are two vertices with multiplicity $(1)$ then, a priori, $S(\Teul)$ is as in Figure \ref{q723roqudh8qhax}(a),
where we have $r\ge2$ dicriticals on one side of the root, and $s\ge2$ dicriticals on the other side.
To prove that $S(\Teul)$ is as in Figure~\ref{273r872qhwudsaku32}, it's enough to show that $\min(r,s)=2$.

\begin{figure}[ht]
\begin{center}
\setlength{\unitlength}{.7mm}%
\begin{picture}(100,57)(-50,-52.5)

\put(-40,0){\circle*{1.7}}
\put(40,0){\circle*{1.7}}
\put(-20,0){\circle{1.7}}
\put(0,0){\circle{1.7}}
\put(20,0){\circle{1.7}}
\put(-20,-35){\circle*{1.7}}
\put(20,-35){\circle*{1.7}}

\put(-40,0){\vector(-1,0){10}}
\put(-40,0){\vector(0,-1){10}}
\put(-20,-35){\vector(-1,0){10}}
\put(-20,-35){\vector(0,-1){10}}

\put(40,0){\vector(1,0){10}}
\put(40,0){\vector(0,-1){10}}
\put(20,-35){\vector(1,0){10}}
\put(20,-35){\vector(0,-1){10}}

\put(-19.1,0){\line(1,0){18.2}}
\put(-39.1,0){\line(1,0){18.2}}
\put(0.9,0){\line(1,0){18.2}}
\put(20.9,0){\line(1,0){18.2}}

\put(20,-.9){\line(0,-1){33.2}}
\put(-20,-.9){\line(0,-1){33.2}}

\put(-40,-20){\circle*{1.7}}
\put(-40,-20){\vector(-1,0){10}}
\put(-40,-20){\vector(0,-1){10}}
\put(-40,-20){\line(1,1){19.3}}
\put(-40,-19){\makebox(0,0)[rb]{\tiny $(0)$}}
\put(-40,-31){\makebox(0,0)[t]{\tiny $(0)$}}
\put(-41,-21){\makebox(0,0)[rt]{\tiny $v_2$}}

\put(40,-20){\circle*{1.7}}
\put(40,-20){\vector(1,0){10}}
\put(40,-20){\vector(0,-1){10}}
\put(40,-20){\line(-1,1){19.3}}
\put(40,-19){\makebox(0,0)[lb]{\tiny $(0)$}}
\put(40,-31){\makebox(0,0)[t]{\tiny $(0)$}}
\put(41,-21){\makebox(0,0)[lt]{\tiny $v_2'$}}

\put(-40,1.5){\makebox(0,0)[b]{\tiny $(0)$}}
\put(-41,-10){\makebox(0,0)[r]{\tiny $(0)$}}
\put(-20,1.5){\makebox(0,0)[b]{\tiny $(1)$}}
\put(20,1.5){\makebox(0,0)[b]{\tiny $(1)$}}
\put(40,1.5){\makebox(0,0)[b]{\tiny $(0)$}}
\put(-20,-46){\makebox(0,0)[t]{\tiny $(0)$}}
\put(-19,-36){\makebox(0,0)[lb]{\tiny $(0)$}}
\put(-21,-36){\makebox(0,0)[rt]{\tiny $v_1$}}
\put(-41,-1){\makebox(0,0)[rt]{\tiny $v_r$}}
\put(21,-36){\makebox(0,0)[lt]{\tiny $v_1'$}}
\put(41,-1){\makebox(0,0)[lt]{\tiny $v_s'$}}

\put(41,-10){\makebox(0,0)[l]{\tiny $(0)$}}
\put(19,-36){\makebox(0,0)[rb]{\tiny $(0)$}}
\put(20,-46){\makebox(0,0)[t]{\tiny $(0)$}}

\put(0,-1.5){\makebox(0,0)[t]{\tiny $v_0$}}
\put(-19,-1){\makebox(0,0)[lt]{\tiny $v$}}
\put(19,-1){\makebox(0,0)[rt]{\tiny $v'$}}

\put(0,-55){\makebox(0,0)[t]{(a) The shadow $S(\Teul)$ of $\Teul$.}}
\end{picture}
\hfill
\begin{picture}(100,57)(-50,-52.5)

\put(-40,0){\circle*{1.7}}
\put(40,0){\circle*{1.7}}
\put(-20,0){\circle{1.7}}
\put(0,0){\circle{1.7}}
\put(20,0){\circle{1.7}}
\put(-20,-35){\circle*{1.7}}
\put(20,-35){\circle*{1.7}}

\put(-40,0){\vector(-1,0){10}}
\put(-40,0){\vector(0,-1){10}}
\put(-20,-35){\vector(-1,0){10}}
\put(-20,-35){\vector(0,-1){10}}

\put(40,0){\vector(1,0){10}}
\put(40,0){\vector(0,-1){10}}
\put(20,-35){\vector(1,0){10}}
\put(20,-35){\vector(0,-1){10}}

\put(-19.1,0){\line(1,0){18.2}}
\put(-39.1,0){\line(1,0){18.2}}
\put(0.9,0){\line(1,0){18.2}}
\put(20.9,0){\line(1,0){18.2}}

\put(20,-.9){\line(0,-1){33.2}}
\put(-20,-.9){\line(0,-1){33.2}}

\put(-40,-20){\circle*{1.7}}
\put(-40,-20){\vector(-1,0){10}}
\put(-40,-20){\vector(0,-1){10}}
\put(-40,-20){\line(1,1){19.3}}
\put(-40,-31){\makebox(0,0)[t]{\tiny $(0)$}}
\put(-41,-21){\makebox(0,0)[rt]{\tiny $v_2$}}
\put(-39,-21){\makebox(0,0)[lt]{\tiny $a_2$}}

\put(40,-20){\circle*{1.7}}
\put(40,-20){\vector(1,0){10}}
\put(40,-20){\vector(0,-1){10}}
\put(40,-20){\line(-1,1){19.3}}
\put(40,-31){\makebox(0,0)[t]{\tiny $(0)$}}
\put(41,-21){\makebox(0,0)[lt]{\tiny $v_2'$}}
\put(39,-21){\makebox(0,0)[rt]{\tiny $a_2'$}}

\put(-41,-10){\makebox(0,0)[r]{\tiny $(0)$}}
\put(-20.5,1.5){\makebox(0,0)[br]{\tiny $v$}}
\put(-17,1.5){\makebox(0,0)[bl]{\tiny $q$}}
\put(17,1.5){\makebox(0,0)[br]{\tiny $q'$}}
\put(20.5,1.5){\makebox(0,0)[bl]{\tiny $v'$}}
\put(-20,-46){\makebox(0,0)[t]{\tiny $(0)$}}
\put(-21,-36){\makebox(0,0)[rt]{\tiny $v_1$}}
\put(-19,-36){\makebox(0,0)[lt]{\tiny $a_1$}}
\put(-41,-1){\makebox(0,0)[rt]{\tiny $v_r$}}
\put(-39,-1){\makebox(0,0)[lt]{\tiny $a_r$}}
\put(21,-36){\makebox(0,0)[lt]{\tiny $v_1'$}}
\put(19,-36){\makebox(0,0)[rt]{\tiny $a_1'$}}
\put(41,-1){\makebox(0,0)[lt]{\tiny $v_s'$}}
\put(39,-1){\makebox(0,0)[rt]{\tiny $a_s'$}}

\put(41,-10){\makebox(0,0)[l]{\tiny $(0)$}}
\put(20,-46){\makebox(0,0)[t]{\tiny $(0)$}}

\put(0,1.5){\makebox(0,0)[b]{\tiny $v_0$}}
\put(-19,-4){\makebox(0,0)[lt]{\tiny $a$}}
\put(19,-4){\makebox(0,0)[rt]{\tiny $a'$}}

\put(0,-55){\makebox(0,0)[t]{(b) The tree $\Teul$.}}

\end{picture}
\caption{}
\label {q723roqudh8qhax}
\end{center}
\end{figure}

The tree $\Teul$ itself is partially depicted in Figure \ref{q723roqudh8qhax}(b) (the only edge decorations that
appear in the picture are those that play a role in the calculation below).
We denote by $v$ and $v'$ the  vertices of multiplicity $(1)$.
Let $q = q( [v_0,v], v)$ and $q' = q( [v_0,v'], v')$.
Let $v_1, \dots, v_r$ be the dicriticals which are linked to $v$ by an edge,
and $v_1', \dots, v_s'$ those which are linked to $v'$.
Let $\epsilon_i$ (resp.\ $\epsilon_i'$) be the unique dead end incident to $v_i$ (resp.\ $v_i'$).
We arrange the labelling so that $a = q( [v,v_1], v )$ is the maximum
of the decorations near $v$  on edges $[v,w]$ with $w>v$, and similarly for $a'  = q( [v',v_1'], v' )$.
Let $a_i = q(\epsilon_i,v_i)$ and $a_i' = q(\epsilon_i',v_i')$.
We set $n=\sum_{i=2}^r a_i$ and $n'=\sum_{i=2}^s a_i'$.
Then $ a, a', a_1, a_1', n, n' \ge 1 $ and we claim that $\min(n,n')=1$.

We have the following set of equations:
\begin{align}
\label {c8924982ywduh}
q(an+a_1) + a(a'n'+a_1') &= N_v = 1 \\
\label {o278c98172egh}
a'(an+a_1) + q'(a'n'+a'_1) &= N_{v'} = 1 .
\end{align}
It follows that $q,q' < 0$.
Let $k=aa'-qq'$.  From \eqref{c8924982ywduh} and \eqref{o278c98172egh}, we obtain
\begin{equation*}
k(a'n'+a_1') = a'-q  \text{\ \ and\ \ }  k(an+a_1) = a-q'
\end{equation*}
so that $k \ge 1$ and:
\begin{align*}
q &= a' - k(a'n'+a_1') = a'(1-kn') - k a_1' \\
q' &= a - k(an+a_1) = a(1-kn) - k a_1 .
\end{align*}
Then $k= aa'-qq' = aa' - [ a'(1-kn') - k a_1' ] [ a(1-kn) - k a_1 ]$, so
$$
aa' ( knn'- n - n') = aa_1'(1-kn) + a' a_1(1-kn') - ka_1a_1' - 1,
$$
which implies that $knn'-n-n'\leq -1$ and hence that $k=1=\min(n,n')$.
As $n \ge r-1$ and $n' \ge s-1$, we get $\min(r,s)=2$.
So $S(\Teul)$ is as depicted in Figure~\ref{273r872qhwudsaku32} and moreover $aa'-qq'=1$.

We showed that, in all cases, $S(\Teul)$ is 
one of Figures \ref{78r273wegadskjh278}, \ref{q723y87c23redhsi} and~\ref{273r872qhwudsaku32}.
For the converse, see~\ref{236xj19832y162egdd}.
\end{proof}

\begin{parag} \label {236xj19832y162egdd}
For each shadow depicted in Figures \ref{78r273wegadskjh278}, \ref{q723y87c23redhsi} and~\ref{273r872qhwudsaku32},
we find all minimally complete Newton trees at infinity satisfying the hypotheses of \ref{8c3r8273d287ed2uq98}
and having that shadow.
We give the solution and leave the verification to the reader.

Consider Figure \ref{78r273wegadskjh278}.
Let $\epsilon_1$ and $\epsilon_1'$ be the dead ends incident to $v_1,v_1'$ respectively.
We seek all $a_1,a_1',q_1,q_1'$ such that by setting 
$$
q(\epsilon_1,v_1) = a_1,\ \  q(\epsilon_1',v_1') = a_1',\ \  q( [v_0,v_1],v_1) = q_1,\ \   q( [v_0,v_1'],v_1') = q_1'
$$
we obtain a tree with the desired shadow.
The complete solution is obtained by choosing relatively prime positive integers $a_1,a_1'$ and setting
$q_1 = -a_1'$ and $q_1' = -a_1$.

\medskip
Consider Figure \ref{q723y87c23redhsi}, where $r\ge2$ is arbitrary.
Let $\epsilon_i$  be the dead end incident to $v_i$ ($1 \le i \le r$) and $\epsilon_1'$ that incident to $v_1'$.
We seek all $a, a_1, \dots, a_r, a_1', q, q_1, \dots, q_r, q_1'$ such that by setting
\begin{align}
\label {817d481762rted}
&\left\{\begin{array}{l}
q( [v,v_1], v ) = a,\ \  q( [v,v_i], v )=1\ (2 \le i \le r),\ \  q( [v_0,v], v ) = q, \\[1mm]
q( \epsilon_i, v_i ) = a_i \text{\ \ and\ \ }  q( [v,v_i], v_i ) = q_i \ \  \text{for $i=1,\dots,r$,}
\end{array} \right. \\[1mm]
& \ \ \ q( \epsilon_1', v_1') = a_1', \ \   q( [v_0,v_1'], v_1' ) = q_1'  
\end{align}
we obtain a tree with the desired shadow. The solution is as follows: 
choose any positive integers $a, a_1, \dots, a_r, a_1'$ such that
$a a_1' \equiv 1 \pmod{an+a_1}$, where $n=\sum_{i=2}^r a_i$.
Then define $q = (1-a a_1')/(an+a_1)$ and let $q_1, \dots, q_r, q_1'$ be given by
$$
q_1 = -q n - a_1',\quad q_i = q a a_i - 1\ (2 \le i \le r),\quad q_1' = -a n - a_1.
$$

\medskip
Consider Figure \ref{273r872qhwudsaku32}, where $r\ge2$ is arbitrary.
Let $\epsilon_i$  be the dead end incident to $v_i$ ($1 \le i \le r$)
and $\epsilon_i'$ that incident to $v_i'$ ($i=1, 2$).
We seek all $a, a_1, \dots, a_r, a', a_1', a_2' q, q_1, \dots, q_r, q_1', q_2'$ such that by setting
\eqref{817d481762rted} and
\begin{equation*}
\left\{ \begin{array}{l}
q( [v',v_1'], v' ) = a',\ \  q( [v',v_2'], v' )=1,\ \  q( [v_0,v'], v' ) = q', \\[1mm]
q( \epsilon_i', v_i' ) = a_i' \text{\ \ and\ \ }  q( [v',v_i'], v_i' ) = q_i' \ \  \text{for $i=1,2$,}
\end{array} \right.
\end{equation*}
we obtain a tree with the desired shadow. The solution is as follows: 
set $a_2'=1$ and choose any positive integers
$a, a_1, \dots, a_r, a', a_1'$ such that $a_1'( a(1-n) - a_1 )+aa' = 1$, where $n=\sum_{i=2}^r a_i$.
Define $q = -a_1'$, $q'= a(1-n)-a_1$ and let $q_1, \dots, q_r, q_1', q_2'$ be given by
$$
\begin{array}{ll}
q_1 = -qn - a' - a_1' & q_i = - aq (n-a_i) - q a_1 - a(a'+a_1')\ (2 \le i \le r), \\[1mm]
q_1' = -q' - (an+a_1) & q_2' = -q'a_1' - a'(an+a_1) .
\end{array}
$$
\end{parag}

\section{Application to rational polynomials of simple type}

Although Neumann and Norbury  \cite{NeumannNorbury:simple} give explicit polynomials, their classification 
is initially presented in terms of the splice diagrams for the link at infinity of a generic fiber of the polynomial.

Consider a primitive polynomial $F(X,Y) \in \Comp[X,Y]$ with at least two points at infinity.
For each $\lambda \in \Comp$, consider the curve
$C_\lambda \subset \aff_\Comp^2$ defined by the equation $F(X,Y)=\lambda$
and let $i : \aff_\Comp^2 \to \proj^2_\Comp$ be the standard embedding $(x,y) \mapsto (x:y:1)$.
By Theorems~2 and 3 of \cite{Neumann:Inv}, the pair $(C_\lambda, i)$ determines a rooted  RPI
splice diagram  $\Omega_\lambda$ for the link of $C_\lambda$ at infinity.  The description appears on p.~448 of
\cite{Neumann:Inv} and properties of the decorations of the edges are given at the end of page 449. 
Replacing each vertex (other than the root) of valency $1$ in $\Omega_\lambda$
by an arrow decorated by $(0)$  produces an abstract Newton tree at infinity $\Teul(\lambda)$.
The multiplicity of a vertex $v$, denoted by $N_v$ here, is denoted in  \cite{Neumann:Inv} by $l_v$ (see Lemma 3.2).
The number of points at infinity of $F(X,Y)$ is equal to the number of points at infinity of $\Teul(\lambda)$.
The important result that we need is Theorem 4.3 of \cite{Neumann:Inv} which implies that the multiplicity $M(\Teul(\lambda))$ is the Euler characteristic of $C_\lambda$.
Recall that for an algebraic curve $C$ 
\begin{equation} \label {f0932d9230}
\chi (C)=2-2g-n
\end{equation}
where $g$ is the genus of $C$ and $n$ its number of places at infinity.

There exists a nonempty Zariski open subset $U$ of $\Comp$ such that
$\Teul(\lambda)$ is generic and independent of $\lambda \in U$ up to equivalence (i.e., up to the equivalence relation defined in Def.\ \ref{83724d91872trhy}).
We denote by $\Teul(F)$ the unique minimally complete Newton tree in this class.
We have $M( \Teul(F) ) = M( \Teul(\lambda) )$ by \ref{0923923ujds0}, so
\begin{equation} \label {f92392e9q089}
M( \Teul(F) ) = \chi (C_\lambda)\ \ \text{for all $\lambda \in U$.}
\end{equation}

Moreover, by \cite[Lemma, p.\ 305]{BartoCassou:RemPolysTwoVars},
there is a bijection between the set of dicriticals of $F : \aff^2_\Comp \to \aff^1_\Comp$
and the set of dicriticals of $\Teul(F)$,
under which corresponding dicriticals have the same degree.

\begin{proposition} \label {jkcnvo2ws9cj}
Let $F \in A = \Comp^{[2]}$ be a rational polynomial of simple type which is not a variable of $A$.
Then there exist $X,Y$ such that $A = \Comp[X,Y]$ and $F(X,Y)$ has two points at infinity.
Choose such a generating pair $X,Y$ for $A$, and define $\Teul(F)$ as in the above paragraphs
(the definition of $\Teul(F)$ depends on the choice of a generating pair).
Then $\Teul(F)$  satisfies all hypotheses of Theorem \ref{8c3r8273d287ed2uq98}.
Consequently, the shadow of $\Teul(F)$ is one of Figures \ref{78r273wegadskjh278},
\ref{q723y87c23redhsi}, \ref{273r872qhwudsaku32}.
\end{proposition}

\begin{proof}
Since $F$ is a rational polynomial which is not a variable of $A$,
\cite[Thm 4.5]{Rus:FieldGen} implies that there exist $X,Y$ such that $A = \Comp[X,Y]$ and
$F(X,Y)$ has two points at infinity. Choose such a generating pair $X,Y$ for $A$, and define $\Teul(F)$ as before.
Then $\Teul(F)$ has two points at infinity. 
Since all dicriticals of $F$ have degree $1$, the dicriticals of $\Teul(F)$ have the same property,
and in particular their gcd is $1$.
If $C_\lambda$ is a generic fiber of  $F$ then the number of places at infinity of $C_\lambda$ is equal to
the number of dicriticals of $F$, which is equal to the number $d$ of dicriticals of $\Teul(F)$;
so \eqref{f0932d9230} gives $ \chi (C_\lambda)=2-d $ and so \eqref{f92392e9q089} gives $ M( \Teul(F) ) = 2-d .$
So $\Teul(F)$ satisfies all hypotheses of Theorem \ref{8c3r8273d287ed2uq98}.
\end{proof}

In the Neumann and Norbury paper \cite{NeumannNorbury:simple},
the splice diagrams are given by Figures 8, 11 and 13 (actually Figure 11 is a special case of Figure 13).
Transforming the splice diagrams in Newton trees, we can see that the shadow of Figure 8 in \cite{NeumannNorbury:simple}
is Figure 4 of the present article, and that the shadow of Figure 13 is Figure 3.
Because their classification restricts itself to {\it ample\/} rational polynomials of simple type,
our Figure~2 does not appear in  \cite{NeumannNorbury:simple}.
By Prop.\ \ref{jkcnvo2ws9cj}, it follows that
{\it the splice diagrams given in  \cite{NeumannNorbury:simple} are correct.}
As explained in the introduction, this needed to be confirmed.

\bibliographystyle{amsplain}
\bibliography{/Users/ddaigle/AAA/articles/bib/dbase}

\end{document}